\documentclass[a4paper]{amsart}

\usepackage[matrix,arrow,curve,tips]{xy}
            \SelectTips{cm}{}
\usepackage{amssymb}

\newtheorem{dfn}{Definition}[section]
\newtheorem{prop}[dfn]{Proposition}
\newtheorem{theo}[dfn]{Theorem}
\newtheorem{ex}[dfn]{Example}

\newcommand{\RR}{\mathbb{R}}
\newcommand{\CC}{\mathbb{C}}

\newcommand{\NN}{\mathbb{N}}

\newcommand{\U}{\mathrm{UFr}}

\newcommand{\cB}{\mathcal{B}}

\newcommand{\ra}{\rightarrow}

\newcommand{\scp}[1]{\left< #1 \right>}
\newcommand{\norm}[1]{\left\lVert #1 \right\rVert}

\newcommand{\com}{\mathbin{{\scriptstyle \circ }}}
\newcommand{\ten}{\mathbin{\otimes}}

\newcommand{\pr}{\mathord{\mathrm{pr}}}

\newcommand{\support}{\mathord{\mathrm{supp_{o}}}}
\newcommand{\uni}{\mathord{\mathrm{uni}}}
\newcommand{\inv}{\mathord{\mathrm{inv}}}
\newcommand{\mlt}{\mathord{\mathrm{mlt}}}

\newcommand{\GPD}{\mathsf{GPD}}

\title[]
      {Representations of orbifold groupoids}

\author{J. Kali\v{s}nik}
\address{Institute of Mathematics, Physics and Mechanics,
         University of Ljubljana, Jadranska 19,
         1000 Ljubljana, Slovenia}
\email{jure.kalisnik@fmf.uni-lj.si}

\thanks{This work was supported in part by
        the Slovenian Ministry of Science}

\begin{document}

\begin{abstract}
Orbifold groupoids have been recently widely used to represent
both effective and ineffective orbifolds. We show that every
orbifold groupoid can be faithfully represented on a continuous
family of finite dimensional Hilbert spaces. As a consequence we
obtain the result that every orbifold groupoid is Morita
equivalent to the translation groupoid of an almost free action of
a proper bundle of topological groups.
\end{abstract}

\maketitle

\section{Introduction}

Orbifolds have generated a lot of interest in the recent
mathematical and physical literature. As first defined in the
paper of Satake \cite{Sat}, under the name of $V$-manifolds, they
generalise the notion of smooth manifolds, by being slightly
singular. More precisely, they are locally homeomorphic to the
space of orbits of a finite group action on some Euclidian space.
The original definition of Satake is equivalent to the modern
definition of an effective orbifold.

The problem of generalising the definition of an orbifold to
incorporate ineffective group actions in local charts is
cumbersome. More convenient way is to use the language of Lie
groupoids, as shown in the work of Moerdijk and Pronk
\cite{Mo,MoPr}. Although the theory of groupoids might seem
abstract and lacking of geometric intuition at first, it provides
a powerfull tool to extend the differential geometric ideas to the
(singular) spaces such as the spaces of leaves of a foliation
\cite{CdSW,Con94}, spaces of orbits of Lie group actions and, in
our case, orbifolds. Orbifold groupoids \cite{Adem,Mo,MoMr,MoPr}
have been effectively used to represent orbifolds in the language
of Lie groupoids. The space of orbits of such a groupoid carries a
natural structure of an orbifold. Moreover, it is easy to describe
effective orbifold groupoids as those groupoids that correspond to
the effective orbifolds. In this way the definition of an
ineffective orbifold comes for free in the framework of orbifold
groupoids.

It is a well known result (see \cite{Adem} or \cite{MoMr} for
details) that the space of orbits of a smooth almost free action
of a compact Lie group on a smooth manifold, such that the slice
representations are effective, carries a natural structure of an
effective orbifold. Conversely, each effective orbifold is
isomorphic to the space of orbits of an almost free action of a
unitary group on the bundle of frames of the orbifold. In the
language of Lie groupoids this statement can be reformulated to
saying that each effective orbifold is Morita equivalent to the
translation groupoid of an almost free action of a compact group
on a smooth manifold. It is conjectured (global quotient
conjecture, see \cite{Adem} for the formulation of the
conjecture), but unknown at present, that similar statement holds
for ineffective orbifolds as well. A partial result was obtained
by Henriques and Metzler in \cite{Hen}, where they proved the
statement for the class of orbifolds, whose ineffective groups
have trivial centre.

The problem of presenting an orbifold groupoid as a translation
groupoid of an almost free action of a compact Lie group is
equivalent to finding a faithfull unitary representation of the
groupoid on some hermitian vector bundle over the space of objects
of the groupoid. We show in this paper (Theorem
\ref{Representation of orbifold groupoids}) that such a
representation (faithfull and unitary) exists on a continuous
family of finite dimensional Hilbert spaces over the space of
objects of the groupoid. As a result (see Theorem \ref{Orbifold
groupoids Morita equivalence}) we can show that each orbifold
groupoid is Morita equivalent to a translation groupoid of an
almost free action of a proper bundle of topological groups, with
each group being a finite product of unitary groups.

\section{Preliminaries}

\subsection{The Morita category of Lie groupoids}

In this section we review the basic definitions and facts that
will be used throughout the paper.

The notion of a topological groupoid is a combination and a
generalization of both topological spaces and topological groups.
The topological part is reflected in the space of orbits of the
groupoid, which carries information of its transversal structure.
On the other hand, the isotropy groups of the groupoid represent
the algebraic part of the groupoid and make it a topological space
with extra algebraic structure. Roughly, two groupoids represent
the same geometric space if they have isomorphic transversal and
algebraic structures. From the differential geometric viewpoint
Lie groupoids form the most interesting class of topological
groupoids and allow a natural extension of many of the operations
on smooth manifolds.

For the convenience of the reader we first recall the notion of a
topological groupoid (see \cite{Mrc96} for more details) and
proceed to the definition of Lie groupoids and generalised maps
between them. Detailed exposition with many examples of Lie
groupoids can be found in one of the books \cite{Mac,MoMr,MoMr2}
and references cited there.

A topological groupoid $G$ over the Hausdorff topological space
$G_{0}$ is given by a structure of a category on the topological
space $G$ with objects $G_{0}$, in which all arrows are invertible
and all the structure maps
\[
\xymatrix{G\times^{s,t}_{G_{0}}G \ar[r]^-{\mlt} &
G \ar[r]^-{\inv} & G \ar@<2pt>[r]^{s} \ar@<-2pt>[r]_{t} &
G_{0} \ar[r]^-{\uni} & G
}
\]
are continuous. The maps $s,t$ and $\mlt$ are required to be open,
while the map $\uni$ is an embedding. If $g\in G$ is any arrow
with source $s(g)=x$ and target $t(g)=y$, and $g'\in G$ is another
arrow with $s(g')=y$ and $t(g')=y'$, then the product
$g'g=\mlt(g',g)$ is an arrow from $x$ to $y'$. The map $\uni$
assigns to each $x\in G_{0}$ the identity arrow $1_{x}=\uni(x)$ in
$G$, and we often identify $G_{0}$ with $\uni(G_{0})$. The map
$\inv$ maps each $g\in G$ to its inverse $g^{-1}$. We use the
notation $G(x,y)=s^{-1}(x)\cap t^{-1}(y)$ for the set of arrows
from $x$ to $y$ and we denote by $G_{x}=G(x,x)$ the isotropy group
of the element $x$.

Each groupoid $G$ induces an equivalence relation on its space of
objects $G_{0}$ by identifying two points if and only if there is
an arrow between them. The resulting quotient map $q:G_{0}\ra
G_{0}/G$ onto the space of orbits is an open surjection. The
subset $O$ of $G_{0}$ is $G$-invariant if it is saturated with
respect to this natural equivalence relation. If $O\subset G_{0}$
is an open subset, the space $G|_{O}=t^{-1}(O)\cap s^{-1}(O)$ has
a natural structure of a topological groupoid over $O$.

We say that the groupoid $G$ is Hausdorff if the space of arrows
$G$ is a Hausdorff topological space. In this paper we will be
mostly interested in proper topological groupoids. A topological
groupoid $G$ is proper if it is Hausdorff and if the map
$(s,t):G\ra G_{0}\times G_{0}$ is a proper continuous map.

A Lie groupoid is a topological groupoid $G$ over $G_{0}$, such
that both $G$ and $G_{0}$ are smooth manifolds and where all the
structure maps are smooth. The maps $s$ and $t$ are required to be
submersions with Hausdorff fibers, to insure the existence of a
smooth manifold structure on the space $G\times^{s,t}_{G_{0}}G$,
while the manifold of objects $G_{0}$ is usually taken to be
Hausdorff and second countable. Here are some basic examples of
topological groupoids.

\begin{ex}\rm \label{Ex_Groupoids}
(i) Each smooth (Hausdorff, second countable) manifold $M$ can be
seen as a Lie groupoid with no nontrivial arrows, where
$G=G_{0}=M$ and where all the structure maps equal the identity
map on the manifold $M$. On the other hand, each Lie group is a
Lie groupoid with only one object and the structure maps induced
from the Lie group structure.

(ii) Let a Lie group $K$ act smoothly from the left on a smooth
(Hausdorff, second countable) manifold $M$. The translation
groupoid $K\ltimes M$ of this action has the manifold $M$ as the
space of objects and the space of arrows equal to $K\times M$. The
source and target maps of the translation groupoid are given by
the formulas $s(k,x)=x$ respectively $t(k,x)=k\cdot x$, while the
multiplication is given by $(k',x')(k,x)=(k'k,x)$ for $x'=k\cdot
x$. The identity and inverse maps are then induced from the group
structure of the Lie group $K$. Translation groupoids associated
to right actions of Lie groups on smooth manifolds can be defined
analogously.

(iii) Let $Q$ be a Hausdorff topological space. A bundle of
topological groups over $Q$ is given by a topological space $U$,
together with an open surjection $r:U\ra Q$, such that each fiber
of the map $r$ has a structure of a topological group and these
structures vary continuously across $Q$. Each such bundle can be
naturally seen as a topological groupoid $G=U$ over the space
$G_{0}=Q$ with the structure maps $s=t=r$ and the maps $\uni$,
$\mlt$ and $\inv$ induced by the group structures on the fibers of
the map $r$. The bundle of topological groups $U$ is locally
trivial if the map $r$ is locally trivial.

(iv) A bundle of topological groups $U$ over $Q$ is proper if it
is proper as a topological groupoid. In this case each fiber is
automatically a compact topological group. The converse is not
always true. Let $U$ be a trivial bundle of nontrivial finite
groups over $\mathbb{R}-\{0\}$ together with the trivial group at
$0\in\mathbb{R}$. This bundle of groups over $\mathbb{R}$ has
compact fibers but it is not a proper bundle of topological
groups.

(v) Let $P$ be a Hausdorff topological space and let a bundle of
topological groups $r:U\ra Q$ act on $P$ from the right along the
map $\phi:P\ra Q$ (see below for the definition of the groupoid
action). The translation groupoid $P\rtimes U$ is a topological
groupoid with the space of arrows $P\times_{Q}U$ over the space of
objects $P$. The structure maps are given by: $t(p,u)=p$,
$s(p,u)=p\cdot u$, $(p,u)(p',u')=(p,uu')$,
$\uni(p)=(p,1_{\phi(p)})$ and $(p,u)^{-1}=(p\cdot u,u^{-1})$ for
$\phi(p)=\phi(p')=r(u)=r(u')$ and $p\cdot u=p'$. If $U$ is a
proper bundle of topological groups it follows that $P\rtimes U$
is a proper topological groupoid.
\end{ex}

Morphisms between Lie groupoids are smooth functors. Two Lie
groupoids $G$ and $H$ are isomorphic if there exist morphisms
$i:G\ra H$ and $j:H\ra G$ of Lie groupoids such that $j\circ
i=id_{G}$ and $i\circ j=id_{H}$. However, in the context of the
representation theory of groupoids the notion of a generalised
morphism or a Hilsum-Skandalis map \cite{MoMr2,Mrc96}, which we
review in the sequel, is more suitable.

A smooth left action of a Lie groupoid $G$ on a smooth manifold
$P$ along a smooth map $\pi\!:P\ra G_{0}$ is a smooth map
$\mu\!:G\times^{s,\pi}_{G_{0}}P\ra P$, $(g,p)\mapsto g\cdot p$,
which satisfies $\pi(g\cdot p)=t(g)$, $1_{\pi(p)}\cdot p=p$ and
$g'\cdot(g\cdot p)=(g'g)\cdot p$, for all $g',g\in G$ and $p\in P$
with $s(g')=t(g)$ and $s(g)=\pi(p)$. We define right actions of
Lie groupoids on smooth manifolds in a similar way.

Let $G$ and $H$ be Lie groupoids. A principal $H$-bundle over $G$
is a smooth manifold $P$, equipped with a left action $\mu$ of $G$
along a smooth submersion $\pi\!:P\ra G_{0}$ and a right action
$\eta$ of $H$ along a smooth map $\phi\!:P\ra H_{0}$, such that
(i) $\phi$ is $G$-invariant, $\pi$ is $H$-invariant and both
actions commute: $\phi(g\cdot p)=\phi(p)$, $\pi(p\cdot h)=\pi(p)$
and $g\cdot(p\cdot h)=(g\cdot p)\cdot h$ for every $g\in G$, $p\in
P$ and $h\in H$ with $s(g)=\pi(p)$ and $\phi(p)=t(h)$, (ii)
$\pi:P\ra G_{0}$ is a principal right $H$-bundle:
$(\pr_{1},\eta)\!:P\times^{\phi,t}_{H_{0}}H\ra
P\times^{\pi,\pi}_{G_{0}}P$ is a diffeomorphism.

A map $f\!:P\ra P'$ between principal $H$-bundles $P$ and $P'$
over $G$ is equivariant if it satisfies $\pi'(f(p))=\pi(p)$,
$\phi'(f(p))=\phi(p)$ and $f(g\cdot p\cdot h)=g\cdot f(p)\cdot h$,
for every $g\in G$, $p\in P$ and $h\in H$ with $s(g)=\pi(p)$ and
$\phi(p)=t(h)$. Any such map is automatically a diffeomorphism.
Principal $H$-bundles $P$ and $P'$ over $G$ are isomorphic if
there exists an equivariant diffeomorphism between them. A
generalised map (sometimes called Hilsum-Skandalis map) from $G$
to $H$ is an isomorphism class of principal $H$-bundles over $G$.

If $P$ is a principal $H$-bundle over $G$ and $P'$ is a principal
$K$-bundle over $H$, for another Lie groupoid $K$, one can define
the composition $P\ten_{H}P'$ \cite{MoMr2,Mrc96,Mrc99}, which is a
principal $K$-bundle over $G$. It is the quotient of
$P\times^{\phi,\pi'}_{H_{0}}P'$ with respect to the diagonal
action of the groupoid $H$. Lie groupoids form a category $\GPD$
\cite{MoMr2,Mrc96} with generalised maps from $G$ to $H$ as
morphisms between groupoids $G$ and $H$. A principal $H$-bundle
$P$ over $G$ is called a Morita equivalence if it is also left
$G$-principal. The isomorphisms in the category $\GPD$ correspond
precisely to equivalence classes of Morita equivalences.

Actions of topological groupoids on topological spaces and the
generalised maps between topological groupoids can be defined in a
similar way. In the topological category all the maps are required
to be continuous, while the condition that the projection map
$\pi:P\ra G_{0}$ of the principal $H$-bundle $P$ is a surjective
submersion is replaced by the condition that $\pi$ is an open
surjective map.

Orbifold groupoids, which are defined in the next subsection, are
examples of Lie groupoids. In Section
\ref{Section_Representation}, where the representation theorem for
orbifold groupoids (Theorem \ref{Representation of orbifold
groupoids}) is proven, we do not need the notion of a more general
topological groupoid. However, presentation of an orbifold
groupoid by a Morita equivalent translation groupoid (Theorem
\ref{Orbifold groupoids Morita equivalence}), associated to an
almost free action of a proper bundle of topological groups, needs
to be done in the topological category.

\subsection{Orbifolds and Lie groupoids}

Orbifolds are topological spaces which generalise the notion of
smooth manifolds in a way that they locally look like quotients of
smooth manifolds by a finite group action. They were first
introduced by Satake in \cite{Sat} under the name of
$V$-manifolds. That original definition is equivalent to the
definition of effective (also called reduced) orbifolds, found in
the modern literature. A certain class of Lie groupoids, called
orbifold groupoids \cite{Mo,MoMr,MoPr}, can be used to represent
effective orbifolds and at the same time provide a way to define
ineffective orbifolds.

Let $Q$ be a topological space. An orbifold chart of dimension $n$
on the space $Q$ is given by a triple $(\tilde{U},G,\phi)$, where
$\tilde{U}$ is a connected open subset of $\RR^{n}$, $G$ is a
finite subgroup of the group $\text{Diff}(\tilde{U})$ of smooth
diffeomorphisms of $\tilde{U}$ and $\phi:\tilde{U}\ra Q$ is an
open map that induces a homeomorphism between $\tilde{U}/G$ and
$U=\phi(\tilde{U})$. An embedding of an orbifold chart
$(\tilde{U},G,\phi)$ into an orbifold chart $(\tilde{V},H,\psi)$
is a smooth embedding $\lambda:\tilde{U}\ra\tilde{V}$ that
satisfies $\psi\circ\lambda=\phi$. The charts $(\tilde{U},G,\phi)$
and $(\tilde{V},H,\psi)$ are compatible, if for any $z\in U\cap V$
there exists an orbifold chart $(\tilde{W},K,\nu)$ with $z\in W$
and embeddings of the chart $(\tilde{W},K,\nu)$ into the charts
$(\tilde{U},G,\phi)$ respectively $(\tilde{V},H,\psi)$. An
orbifold atlas (of dimension $n$) on $Q$ is given by a family
$\mathcal{U}=\{(U_{i},G_{i},\phi_{i})\}_{i\in I}$ of pairwise
compatible orbifold charts (of dimension $n$) that cover $Q$. An
atlas $\mathcal{U}$ refines the atlas $\mathcal{V}$ if every chart
of $\mathcal{U}$ can be embedded into some chart of $\mathcal{V}$.
Two orbifold atlases are equivalent if there exists an atlas that
refines both of them. Effective orbifold of dimension $n$ is a
paracompact Hausdorff topological space $Q$ together with an
equivalence class of $n$-dimensional orbifold atlases on $Q$.

Primary examples of effective orbifolds are the orbit spaces of
effective actions of finite groups on smooth manifolds, where the
charts are given by connected components of (small enough)
invariant open subsets. More generally (see \cite{Adem,MoMr}), let
a compact Lie group $K$ act smoothly and almost freely (with
finite isotropy groups) on a smooth manifold $M$. Since the
actions of compact Lie groups are proper, there exist local
slices, equipped with the actions of the isotropy groups. If these
actions are assumed to be effective, the slices can be used as the
local orbifold charts on the space of orbits $M/K$.

The group actions, defined by the charts, and especially the
isotropy groups form an important part of the orbifold structure.
Namely, two orbifolds can be non-isomorphic, despite being
homeomorphic, when seen as topological spaces. Since Lie groupoids
have a natural built-in algebraic structure, they provide a
suitable framework for the study of orbifolds \cite{MoMr,MoPr}.
Furthermore, the notion of a generalised morphism between Lie
groupoids representing orbifolds turns out to be the proper notion
of a map between the corresponding orbifolds.

Let $G$ be a Lie groupoid. If the maps $s$ and $t$ (and therefore
all structure maps) are local diffeomorphisms we call $G$ an
\'{e}tale Lie groupoid. A bisection of an \'{e}tale Lie groupoid
$G$ is an open subset $U$ of $G$ such that both $s|_{U}$ and
$t|_{U}$ are injective. Any such bisection $U$ gives a local
diffeomorphism $\tau_{U}\!:s(U)\ra t(U)$, by $\tau_{U}=t|_{U}\com
(s|_{U})^{-1}$. For any arrow $g\in G(x,y)$ there exists a
bisection $U_{g}$ containing $g$; the germ at $x$ of the induced
local diffeomorphism is independent of the choice of the
bisection.

Orbifold groupoid is a proper \'{e}tale Lie groupoid. An orbifold
groupoid $G$ is effective if for each $x\in G_{0}$ and each
nontrivial $g\in G_{x}$ the germ at $x$ of some (and therefore
every) local diffeomorphism $\tau_{U_{g}}$, defined by a bisection
through $g$, is nontrivial.

Crucial theorem in the connection between effective orbifolds and
effective orbifold groupoids states that there is a natural
structure of an effective orbifold \cite{MoMr,MoPr} on the space
of orbits of an effective orbifold groupoid. In this way an
ineffective orbifold groupoid can be seen as one possible way to
define an ineffective orbifold.

\subsection{Continuous families of Hilbert spaces}

Representation theory of topological groupoids extends the
classical representation theory of groups on vector spaces, where
the latter are replaced by families of vector spaces, indexed by
the space of objects of the groupoid. We first recall the
definition and basic properties of a continuous family of Hilbert
spaces over a topological space, as given in \cite{Dix} (see also
\cite{Bos} for further examples).

\begin{dfn}\label{Family of Hilbert spaces}
Let $B$ be a locally compact Hausdorff topological space. A
continuous family of Hilbert spaces over $B$ is given by a pair
$(\{E_{x}\}_{x\in B},\Gamma)$, where $E_{x}$ is a Hilbert space
for each $x\in B$ and $\Gamma\subset\prod_{x\in B}E_{x}$ is a
vector subspace that satisfies:
\begin{enumerate}
\item For each $x\in B$ and each $v\in E_{x}$ there exists
$s\in\Gamma$ such that $s(x)=v$;

\item For every $s_{1},s_{2}\in\Gamma$ the function $x\mapsto
\scp{s_{1}(x),s_{2}(x)}_{x}$ is a continuous function on $B$;

\item If $w\in\prod_{x\in B}E_{x}$ satisfies: for each $x\in B$
and each $\epsilon>0$ there exists a neighbourhood $U$ of $x$ and
$s\in\Gamma$ such that $\norm{s(x')-w(x')}_{x'}<\epsilon$ for all
$x'\in U$, then $w\in\Gamma$.
\end{enumerate}
Family $(\{E_{x}\}_{x\in B},\Gamma)$ is a continuous family of
finite dimensional Hilbert spaces if all the Hilbert spaces
$E_{x}$ are finite dimensional.
\end{dfn}

From the topological viewpoint the following consequence of
Definition \ref{Family of Hilbert spaces} is useful and allows us
to think of continuous families of Hilbert spaces as
generalizations of hermitian vector bundles.

\begin{prop}
Let $(\{E_{x}\}_{x\in B},\Gamma)$ be a continuous family of
Hilbert spaces over a locally compact Hausdorff space $B$. Denote
by $E=\coprod_{x\in B} E_{x}$ the disjoint union of the spaces
$\{E_{x}\}_{x\in B}$. There exists a topology on the space $E$
that makes the projection map $p:E\ra B$ (which maps each Hilbert
space $E_{x}$ to the point $x$) a continuous open surjection and
such that the space $\Gamma$ equals the space of continuous
sections of the map $p$.
\end{prop}

\begin{proof}
We first define a basis for the topology on the total space $E$.
For each open subset $V\subset B$, each $s\in\Gamma$ and each
$\epsilon>0$ define the tubular set $B(V,s,\epsilon)=\{v\in
E|p(v)\in V,\norm{s(p(v))-v}_{p(v)}<\epsilon\}$. Condition $(1)$
in Definition \ref{Family of Hilbert spaces} insures that the
family of all such tubular sets covers the space $E$. Now let
$W_{1}=B(V_{1},s_{1},\epsilon_{1})$ and
$W_{2}=B(V_{2},s_{2},\epsilon_{2})$ be two such tubular sets and
choose arbitrary element $v\in W_{1}\cap W_{2}$. For any such $v$
the inequalities $\norm{s_{i}(p(v))-v}_{p(v)}<\epsilon_{i}$ hold
for $i=1,2$ and there exists a section $s\in\Gamma$ such that
$s(p(v))=v$. Denote
$\delta=\min\{\epsilon_{1}-\norm{s_{1}(p(v))-v}_{p(v)},
\epsilon_{2}-\norm{s_{2}(p(v))-v}_{p(v)}\}$. Since $\Gamma$ is a
vector subspace of $\prod_{x\in B}E_{x}$, $s_{1}-s$ and $s_{2}-s$
are elements of $\Gamma$ as well. Using condition $(2)$ in
Definition \ref{Family of Hilbert spaces} we can find open
neighbourhoods $U_{1}$ and $U_{2}$ of the point $p(v)$ such that
$\norm{s_{i}(x)-s(x)}_{x}<\epsilon_{i}-\frac{\delta}{2}$ for
$i=1,2$ and all $x\in U_{1}$ respectively $x\in U_{2}$. The
tubular set $B(U_{1}\cap U_{2},s,\frac{\delta}{2})$ then satisfies
$B(U_{1}\cap U_{2},s,\frac{\delta}{2})\subset W_{1}\cap W_{2}$ and
contains the point $v$.

With the above topology the map $p$ becomes a continuous open
surjection. It remains to be proven that the space $\Gamma$ equals
the space of the continuous sections of the map $p$. Choose any
section $s\in\Gamma$, sending $x\in B$ to $v\in E$. We want to
show that $s$ is a continuous section of the map $p$. For any
basic open neighbourhood $B(V,s',\epsilon)$ of the element $v$ we
have $s-s'\in\Gamma$ and $\norm{s(x)-s'(x)}_{x}<\epsilon$.
Continuity of the map $y\mapsto\norm{s(y)-s'(y)}_{y}$ gives us a
neighbourhood $U$ of the point $x$ such that
$\norm{s(y)-s'(y)}_{y}<\epsilon$ on $U$. The neighbourhood $U\cap
V$ then satisfies $s(U\cap V)\subset B(V,s',\epsilon)$, which
proves that $s$ is a continuous section of the map $p$.
Conversely, let $s:B\ra E$ be any continuous section of the map
$p$. We will show that $s$ satisfies condition $(3)$ in Definition
\ref{Family of Hilbert spaces}. Choose an element $x\in B$ and
$\epsilon>0$. By condition $(1)$ in Definition \ref{Family of
Hilbert spaces} we can find $s'\in\Gamma$ such that $s(x)=s'(x)$.
Since $s:B\ra E$ is a continuous map, the set
$V=s^{-1}(B(B,s',\epsilon))$ is an open neighbourhood of the point
$x$ such that $\norm{s(y)-s'(y)}_{y}<\epsilon$ for every $y\in V$.
The neighbourhood with this property exists for every $x\in B$ and
every $\epsilon>0$, therefore $s\in\Gamma$.
\end{proof}

From now on we will denote the continuous family of Hilbert spaces
$(\{E_{x}\}_{x\in B},\Gamma)$ over $B$ simply by $E$, according to
the notations from the preceding proposition, and refer to
$\Gamma$ as the space of the continuous sections of the map $p$.
It is not hard to check that the space $\Gamma$ is in fact a
module over the algebra of the continuous functions on the space
$B$. The dimension $d(x)$ of the fiber $E_{x}$ of a continuous
family of Hilbert spaces is not necesarilly constant along $B$,
but it is a lower semi-continuous function on $B$, as can be seen
by using properties $(1)$ and $(2)$ of Definition \ref{Family of
Hilbert spaces}. Denote by $\support(E)=\{x\in B|d(x)>0\}$ the
support of the family of Hilbert spaces $E$. Notice that
$\support(E)$ is an open subset of $B$ since the dimension
function $d$ is lower semi-continuous.

Here are some examples of continuous families of Hilbert spaces
that will be used later on in the paper.

\begin{ex}\rm \label{Ex_Families of Hilbert spaces}
(i) Every $n$-dimensional hermitian vector bundle $E$ over a
locally compact Hausdorff space $B$ is an example of a family of
finite dimensional Hilbert spaces with fibers of constant
dimension. Conversely, if the dimension of the fibers of the
continuous family of Hilbert spaces $E$ over $B$ is a constant
function on $B$, then $E$ is actually a hermitian vector bundle
over $B$.

(ii) Let $B$ be a locally compact Hausdorff topological space,
$O\subset B$ an open subset and $E_{O}$ a hermitian vector bundle
over the space $O$. The trivial extension of the bundle $E_{O}$ is
the continuous family of Hilbert spaces $E_{O}^{B}$ over $B$,
defined as follows. The fiber of $E_{O}^{B}$ over the point $x\in
B$ is by definition
\begin{displaymath}
(E^{B}_{O})_{x}=\left\{ \begin{array} {ll}
(E_{O})_{x}, & \textrm{$x\in O$,}\\
\{0\}, & \textrm{otherwise.}\\
\end{array}\right.
\end{displaymath}
We define the vector space $\Gamma(E^{B}_{O})$ of sections of
$E^{B}_{O}$ to be the trivial extensions of those sections of
$E_{O}$ that tend to zero at the boundary of the space $O$ in $B$.
By definition, the section $s\in\Gamma(E_{O})$ tends to zero at
the boundary of the space $O$ in $B$ if for every $x\in\partial O$
and every $\epsilon>0$ there exists a neighbourhood $U$ of $x$ in
$B$ such that $\norm{s(y)}_{y}<\epsilon$ for all $y$ in $U\cap O$.
It is straightforward to check that the space $\Gamma(E^{B}_{O})$
satisfies the conditions in Definition \ref{Family of Hilbert
spaces}.

(iii) Let $\{E^{i}\}_{i\in I}$ be a collection of families of
finite dimensional Hilbert spaces over the locally compact
Hausdorff space $B$ and assume that the family of open sets
$\{\support(E^{i})\}_{i\in I}$ is locally finite over $B$. The sum
$E=\bigoplus_{i\in I}E^{i}$ of families $\{E^{i}\}_{i\in I}$ is
defined as follows. First define $E_{x}=\bigoplus_{i\in
I}E^{i}_{x}$ for each $x\in B$. Since the family
$\{\support(E^{i})\}_{i\in I}$ is locally finite, this sum is
actually a finite sum, so $E_{x}$ is a finite dimensional Hilbert
space for every $x\in B$. The space of sections $\Gamma(E)$ is
defined to be the product of the spaces $\Gamma(E^{i})$. The
induced topology on the space $E$ coincides with the topology of
the fibrewise product of the spaces $E^{i}$ along the space $B$.
\end{ex}

\section{Representations of topological groupoids}

Let $G$ be a topological groupoid with locally compact space of
objects $G_{0}$ and let $p:E\ra G_{0}$ be a continuous family of
Hilbert spaces over $G_{0}$. A continuous representation of the
groupoid $G$ on the family $E$ is given by a continuous left
action of $G$ on the space $E$, along the map $p$, such that each
$g\in G(x,y)$ acts as a linear isomorphism $g:E_{x}\ra E_{y}$.
Representation of the groupoid $G$ on the family of Hilbert spaces
$E$ is unitary if each $g\in G$ acts as a unitary map between the
corresponding Hilbert spaces.

\begin{ex}\rm
(i) Let $K$ be a topological group. Then the continuous
representations of $K$, viewed as a topological groupoid, coincide
with the continuous representations of $K$ on Hilbert spaces. Each
continuous family of Hilbert spaces over the locally compact space
$X$ is naturally a representation of the space $X$, seen as a
topological groupoid.

(ii) Combining previous two examples we get the representations of
the translation groupoid $K\ltimes X$ of a continuous action of
the topological group $K$ on the locally compact Hausdorff space
$X$. These are precisely $K$-equivariant continuous families of
Hilbert spaces over $X$, i.e. there is a fibrewise linear action
of the group $K$ on the total space $E$ such that the projection
map $p$ is $K$-equivariant.
\end{ex}

Generalised maps between groupoids can be used to pull back
representations in the same sense as vector bundles can be pulled
back by continuous maps. Let $G$ and $H$ be Lie groupoids and let
$P$ be a principal $H$-bundle over $G$. Assume that $E$ is a
hermitian vector bundle over $H_{0}$, equipped with a unitary
representation of the groupoid $H$, and denote by $\pi:P\ra G_{0}$
respectively $\phi:P\ra H_{0}$ the moment maps of the principal
bundle $P$. The pull back bundle $\phi^{\ast}E=P\times_{H_{0}}E$
has a natural structure of a vector bundle over $P$ with
projection onto the first factor as the projection map. The
groupoid $H$ acts from the right on the space $\phi^{\ast}E$ by
the formula: $(p,v)\cdot h=(p\cdot h,h^{-1}\cdot v)$. Since the
action of $H$ on $P$ is along the fibers of the map $\pi$, it is
easy to see that the map $\pi_{G}:\phi^{\ast}E/H\ra G_{0}$,
$\pi_{G}([p,v])=\pi(p)$ is well defined and continuous. We will
show that the space $P^{\ast}E=\phi^{\ast}E/H$ has a natural
structure of a hermitian vector bundle over $G_{0}$ with
projection map $\pi_{G}$. Furthermore, the action of the groupoid
$G$ on the space $P$ induces a unitary representation of the
groupoid $G$ on the bundle $P^{\ast}E$.

\begin{prop}\label{Pull back bundles}
The space $P^{\ast}E$ is a hermitian vector bundle over $G_{0}$
with a natural unitary representation of the Lie groupoid $G$.
\end{prop}

\begin{proof}
First consider the induced vector bundle $\phi^{\ast}E$ over $P$.
Its fiber over the point $p\in P$ can be canonically identified to
the fiber of the vector bundle $E$ over the point $\phi(p)\in
H_{0}$, while the formula
$\scp{(p,v_{1}),(p,v_{2})}_{\phi^{\ast}E}=\scp{v_{1},v_{2}}_{E}$
induces a scalar product on the fiber $(\phi^{\ast}E)_{p}$. So
defined structures of Hilbert spaces on the fibers of the bundle
$\phi^{\ast}E$ over $P$ turn it into a hermitian vector bundle.

The groupoid $H$ acts from the right on the space $\phi^{\ast}E$
by $(p,v)\cdot h=(p\cdot h,h^{-1}\cdot v)$ for $\phi(p)=t(h)$ and
$v\in E_{t(h)}$. Let $r:\phi^{\ast}E\ra G_{0}$ be the
$H$-invariant projection defined with the formula
$r((p,v))=\pi(p)$ and denote by $\pi_{G}:\phi^{\ast}E/H\ra G_{0}$
the induced map from the quotient space. First observe that the
fibers of the map $r:\phi^{\ast}E\ra G_{0}$ equal the restrictions
$P^{\ast}E|_{\pi^{-1}(x)}$ of the bundle $P^{\ast}E$ to the fibers
of the map $\pi$ over the points $x\in G_{0}$. Furthermore, since
$P$ is a principal $H$-bundle over $G_{0}$, $H$ acts freely and
transitively along the fibers of the map $\pi$. Combining these
two observations with the fact that the action of $H$ on $E$ is
linear we get natural structures of Hilbert spaces on the fibers
of the map $\pi_{G}$, which we now describe. Let
$\delta=pr_{2}\circ(\pr_{1},\eta)^{-1}\!: P\times_{G_{0}}P\ra H$
be the continuous (actually smooth in our case) map, uniquely
defined by the condition $p\cdot\delta(p,p')=p'$, for $p$ and $p'$
that satisfy $\pi(p)=\pi(p')$. Using the map $\delta$ we can
define the maps $+:P^{\ast}E\times_{G_{0}}P^{\ast}E\ra P^{\ast}E$,
$\cdot:\CC\times P^{\ast}E\ra P^{\ast}E$ and
$\scp{-,-}_{P^{\ast}E}:P^{\ast}E\times_{G_{0}}P^{\ast}E\ra\CC$ by
the formulas
\[
[p,v]+[p',v']=[p,v+\delta(p,p')v'],
\]
\[
\lambda[p,v]=[p,\lambda v],
\]
\[
\scp{[p,v],[p',v']}_{P^{\ast}E}=\scp{v,\delta(p,p')v'}_{E}.
\]
It is straightforward to verify that these maps are well defined,
continuous and that they induce structures of Hilbert spaces on
the fibers of the map $\pi_{G}$. To show that the map
$\pi_{G}:P^{\ast}E\ra G_{0}$ carries a structure of a vector
bundle we have to show that it is locally trivial. Choose a point
$x\in G_{0}$ and an element $p\in P$ with $\pi(p)=x$. Let
$\psi:\phi^{\ast}E|_{V}\ra V\times\CC^{n}$ be a trivialization of
the bundle $\phi^{\ast}E$ on the neighbourhood $V$ of the point
$p$. Since the map $\pi$ is a submersion, there exists a local
section $s:U\ra V$ of the map $\pi|_{V}$, defined on some
neighbourhood $U$ of the point $x$, such that $s(x)=p$. Using the
maps $\psi$ and $s$ we can define the trivialization
$\psi':P^{\ast}E|_{U}\ra U\times\CC^{n}$ by the formula
$\psi'([p',v'])=(\pi(p'),pr_{2}(\psi(s(\pi(p')),\delta(s(\pi(p')),p')v'))$.
Finally, by defining $g\cdot[p,v]=[g\cdot p,v]$, we get a unitary
representation of the Lie groupoid $G$ on the hermitian vector
bundle $P^{\ast}E$.
\end{proof}

A representation of the groupoid $G$ on the family of Hilbert
spaces $E$ is faithfull if for each $x\in G_{0}$ the isotropy
group $G_{x}$ acts faithfully on the Hilbert space $E_{x}$. This
condition is equivalent to the requirement that for each $x,y\in
G_{0}$ and $g_{1},g_{2}\in G(x,y)$, with $g_{1}\neq g_{2}$, the
elements $g_{1}$ and $g_{2}$ induce different isomorphisms from
$E_{x}$ to $E_{y}$.

\begin{prop}\label{Faithfullness}
Let $G$ and $H$ be Lie groupoids and let $P$ be a Morita
equivalence between $G$ and $H$. If the representation of the
groupoid $H$ on $E$ is faithfull the representation of the
groupoid $G$ on $P^{\ast}E$ is faithfull as well.
\end{prop}

\begin{proof}
We have to prove that for each $x\in G_{0}$ the isotropy group
$G_{x}$ acts faithfully on the vector space $(P^{\ast}E)_{x}$.
Choose $x\in G_{0}$ and any $p\in\pi^{-1}(x)\subset P$. Since $P$
is a Morita equivalence, the Lie groups $G_{x}$ and $H_{\phi(p)}$
act freely and transitively on the space
$P(p)=\pi^{-1}(x)\cap\phi^{-1}(\phi(p))$ from the left
respectively from the right. Denote by $i_{p}:G_{x}\ra
H_{\phi(p)}$ the induced bijection (which is in fact a group
isomorphism), implicitely defined by the equation $g\cdot p=p\cdot
i_{p}(g)$. Suppose that an arrow $g\in G_{x}$ acts as the identity
transformation on the space $(P^{\ast}E)_{x}$, i.e.
$g\cdot[p,v]=[p,v]$ for all $v\in E_{\phi(p)}$, where we have
identified $[p,v]\in (P^{\ast}E)_{x}$ with $v\in E_{\phi(p)}$.
Then the equality
\[
[p,v]=g\cdot[p,v]=[g\cdot p,v]=[p\cdot i_{p}(g),v]=[p,i_{p}(g)v]
\]
holds for all $v\in E_{\phi(p)}$. This shows that $i_{p}(g)$ acts
as the identity on the space $E_{\phi(p)}$ and is therefore by the
assumption of faithfullness of the representation of $H$ on $E$
equal to $1_{\phi(p)}$. Since $i_{p}$ is a group isomorphism $g$
must be equal to $1_{x}$, which shows that the representation of
$G_{x}$ on $(P^{\ast}E)_{x}$ is faithfull as well.
\end{proof}

\section{Representations of orbifold groupoids}\label{Section_Representation}

The problem of representing an orbifold groupoid $G$ faithfully
and unitarily on a hermitian vector bundle is equivalent to
finding a smooth almost free action of a compact Lie group $K$ on
a smooth manifold $M$ such that the translation groupoid $M\rtimes
K$ is Morita equivalent to the groupoid $G$.

A faithfull unitary representation of the groupoid $G$ on an
$n$-dimensional hermitian vector bundle $E$ over $G_{0}$ induces a
free left action of the groupoid $G$ on the principal
$U(n)$-bundle $\U(E)$ of unitary frames of the bundle $E$, which
commutes with the natural right action of the Lie group $U(n)$ on
the bundle $\U(E)$. Since the action of $G$ on $\U(E)$ is proper
and free, the orbit space $G\backslash \U(E)$ inherits a natural
smooth structure. Moreover, since the actions of $G$ and $U(n)$ on
$\U(E)$ commute, there exists an induced action of the group
$U(n)$ on the manifold $G\backslash \U(E)$, which is almost free
as a consequence of the fact that $G$ is an orbifold groupoid. It
is then straightforward to check that $G$ is Morita equivalent to
the translation groupoid $(G\backslash \U(E))\rtimes U(n)$.

On the other hand, let the compact Lie group $K$ act smoothly and
almost freely from the right on the smooth manifold $M$ and let
$P$ denote the Morita equivalence between the groupoids $G$ and
$M\rtimes K$. By the Peter-Weyl theorem for compact Lie groups
there exists a finite dimensional Hilbert space $V$ and a
faithfull unitary representation of the group $K$ on $V$. This
representation induces a faithfull unitary representation of the
groupoid $M\rtimes K$ on the trivial vector bundle $M\times V$,
where the action is given by $(x,g)\cdot (x',v)=(x,g\cdot v)$ for
$x'=x\cdot g$. Combining Propositions \ref{Pull back bundles} and
\ref{Faithfullness} we get a faithfull unitary representation of
the groupoid $G$ on the hermitian vector bundle $P^{\ast}(M\times
V)$ over $G_{0}$.

The question whether every orbifold groupoid admits a faithfull
unitary representation on a hermitian vector bundle is believed to
have a positive answer, but it is unproven at the moment. It has
been long known to be true for effective orbifold groupoids, as
sketched below. Each \'{e}tale Lie groupoid $G$ has a natural
representation on the tangent bundle $TG_{0}$, where an arrow
$g\in G(x,y)$ acts via the differential of the local
diffeomorphism $\tau_{U_{g}}$, induced by some bisection $U_{g}$
containing the arrow $g$. Straight from the definition it follows
that this representation of the groupoid $G$ on $TG_{0}$ is
faithfull if and only if $G$ is an effective orbifold groupoid.
This canonical representation can be extended to the
representation of $G$ on the complexified tangent bundle
$T^{\CC}G_{0}$ and made unitary by averaging an arbitrary
hermitian metric on $T^{\CC}G_{0}$. More recently, in the paper
\cite{Hen} by Henriques and Metzler, the authors proved the
statement for the class of ineffective orbifold groupoids, whose
ineffective isotropy groups have trivial centre.

However, in the broader framework of unitary representations on
continuous families of finite dimensional Hilbert spaces we are
able to prove the following result.

\begin{theo}\label{Representation of orbifold groupoids}
Let $G$ be an orbifold groupoid over $G_{0}$. Then there exists a
faithfull unitary representation of the groupoid $G$ on a
continuous family of finite dimensional Hilbert spaces over
$G_{0}$.
\end{theo}

We start by proving some propositions that will be needed in the
proof of Theorem \ref{Representation of orbifold groupoids}.

\begin{prop}\label{Invariant}
Let $G$ be an orbifold groupoid over $G_{0}$. For each $x\in
G_{0}$ there exist a $G$-invariant open neighbourhood $O_{x}$ of
$x$ and a faithfull unitary representation of the groupoid
$G|_{O_{x}}$ on a hermitian vector bundle $E_{O_{x}}$ over
$O_{x}$.
\end{prop}

\begin{proof}
In the proof of the proposition we use the following
characterization of the local structure of orbifold groupoids
\cite{MoMr,MoPr}. For each $x\in G_{0}$ there exist a
neighbourhood $U_{x}$ of $x$ and a natural isomorphism of Lie
groupoids $G|_{U_{x}}\cong G_{x}\ltimes U_{x}$, where each $g\in
G_{x}$ acts on $U_{x}$ by the diffeomorphism corresponding to the
suitable bisection through $g$. Let $\CC[G_{x}]$ denote the
Hilbert space of complex functions on the finite group $G_{x}$
with the orthonormal basis $\{\delta_{g}\}_{g\in G_{x}}$. The left
regular representation of the group $G_{x}$ on the space
$\CC[G_{x}]$ induces a faithfull unitary representation of the
groupoid $G|_{U_{x}}\cong G_{x}\ltimes U_{x}$ on the trivial
vector bundle $U_{x}\times\CC[G_{x}]$ by the formula
$(g,x)\cdot(x,f)=(g\cdot x,g\cdot f)$.

The saturation $O_{x}=s(t^{-1}(U_{x}))$ of the open set $U_{x}$ is
again an open set since $s$ is a submersion and hence an open map.
It is straightforward to check that the manifold
$P=t^{-1}(U_{x})$, together with the left action of the groupoid
$G|_{U_{x}}$ and the right action of the groupoid $G|_{O_{x}}$,
defines a Morita equivalence between the groupoids $G|_{U_{x}}$
and $G|_{O_{x}}$. Denote by
$E_{O_{x}}=(P^{-1})^{*}(U_{x}\times\CC[G_{x}])$ the pullback
bundle over $O_{x}$, together with the induced unitary
representation of the groupoid $G|_{O_{x}}$. Since the
representation of the groupoid $G|_{U_{x}}$ on
$U_{x}\times\CC[G_{x}]$ was faithfull and since $P^{-1}$ is a
Morita equivalence the representation of $G|_{O_{x}}$ on
$E_{O_{x}}$ is faithfull by Proposition \ref{Faithfullness}.
\end{proof}

\begin{prop}\label{Extension}
Let $G$ be an orbifold groupoid over $G_{0}$ and let $O\subset
G_{0}$ be a $G$-invariant open subset of $G_{0}$. Every unitary
representation of the groupoid $G|_{O}$ on a hermitian vector
bundle $E_{O}$ over $O$ can be extended to a unitary
representation of the groupoid $G$ on the continuous family of
finite dimensional Hilbert spaces $E^{G_{0}}_{O}$ over $G_{0}$.
\end{prop}

\begin{proof}
Let $E_{O}$ be a Hermitian vector bundle over the space $O$,
equipped with a unitary representation of the groupoid $G|_{O}$.
Denote by $p:E^{G_{0}}_{O}\ra G_{0}$ the trivial extension of the
hermitian vector bundle $E_{O}$ over $O$ to a family of finite
dimensional Hilbert spaces over $G_{0}$ as in Example
\ref{Ex_Families of Hilbert spaces}. Recall that
\begin{displaymath}
(E^{G_{0}}_{O})_{x}=\left\{ \begin{array} {ll}
(E_{O})_{x}, & \textrm{if $x\in O$,}\\
\{0\}, & \textrm{otherwise.}\\
\end{array}\right.
\end{displaymath}
For any arrow $g\in G$ we define the action as follows:

(i) If $g\in G|_{O}$ let $g$ act on $E^{G_{0}}_{O}$ as it acts on
$E_{O}$;

(ii) If $g\notin G|_{O}$ and $g\in G(x,y)$ then $g$ acts in the
only possible way, sending the vector $0_{x}$ to the vector
$0_{y}$.

This defines a unitary representation of the groupoid $G$ on the
family of finite dimensional Hilbert spaces $E^{G_{0}}_{O}$, which
extends the representation of the groupoid $G|_{O}$ on the vector
bundle $E_{O}$. To prove the claim of the proposition we have to
check that this defines a continuous representation, i.e. the map
$\mu:G\times_{G_{0}}E^{G_{0}}_{O}\ra E^{G_{0}}_{O}$ is continuous.

First decompose the space $G_{0}$ as a disjoint union of
$G$-invariant subspaces $O$, $V=\overline{O}^{c}$ and $\partial
O$. Since $O$ and $V$ are open subsets of the space $G_{0}$, the
spaces $W_{1}=G\times_{G_{0}}p^{-1}(V)$ respectively
$W_{2}=G\times_{G_{0}}p^{-1}(O)$ are open subspaces of the space
$G\times_{G_{0}}E^{G_{0}}_{O}$. Observing that $\mu|_{W_{1}}$ is
basically the left action of the groupoid $G|_{V}$ on $V$, while
$\mu|_{W_{2}}$ equals the action map of the representation of the
groupoid $G|_{O}$ on the bundle $E_{O}$, we see that
$\mu|_{W_{1}}$ respectively $\mu|_{W_{2}}$ are continuous maps.
Now let $g\in G(x,y)$ be an arrow such that $x\in\partial O$ and
therefore $y\in\partial O$. For such $g$ there exists only one
element in $G\times_{G_{0}}E^{G_{0}}_{O}$ with first coordinate
$g$, namely $(g,0_{x})$ and we have $\mu(g,0_{x})=g\cdot
0_{x}=0_{y}$. We need to show that the map $\mu$ is continuous at
the point $(g,0_{x})$. To this extent choose arbitrary
neighbourhood $W$ of the point $0_{y}$ in $E^{G_{0}}_{O}$. By the
definition of the topology on the space $E^{G_{0}}_{O}$ we can
find a smaller tubular open neighbourhood $B(U_{y},0,\epsilon)$ of
the point $0_{y}$, where $U_{y}$ is a neighbourhood of the point
$y$ in $G_{0}$ and $0$ is the zero section of $E^{G_{0}}_{O}$.
Shrinking the set $U_{y}$ if necessary we can assume that there
exists a bisection $U$ of the groupoid $G$ through the arrow $g$
such that $t(U)=U_{y}$. The unitarity of the representation of $G$
on $E^{G_{0}}_{O}$ now implies that
$\mu(U\times_{G_{0}}B(s(U),0,\epsilon))\subset
B(U_{y},0,\epsilon)$, which proves that $\mu$ is continuous at the
point $(g,0_{x})$.
\end{proof}

\begin{proof}[Proof of Theorem \ref{Representation of orbifold groupoids}]
Let $G$ be an orbifold groupoid over $G_{0}$. The quotient
projection $q:G_{0}\ra G_{0}/G$ is an open surjective map, which
insures that the space $Q=G_{0}/G$ is second countable and locally
compact. Since $G$ is a proper groupoid, the map $(s,t):G\ra
G_{0}\times G_{0}$ is a proper map between Hausdorff topological
spaces and hence a closed map. This shows that $(s,t)(G)\subset
G_{0}\times G_{0}$ is a closed equivalence relation, so $Q$ is a
Hausdorff space. It follows that $Q$ is paracompact.

We can use Proposition \ref{Invariant} to find for each $x\in
G_{0}$ a $G$-invariant open neighbourhood $O_{x}$ of the point $x$
and a faithfull unitary representation of the groupoid
$G|_{O_{x}}$ on a Hermitian vector bundle $E_{O_{x}}$ over
$O_{x}$. The family $\{q(O_{x})\}_{x\in G_{0}}$ is an open cover
of the second countable paracompact space $Q$, so we can choose a
countable, locally finite refinement $\{V'_{i}\}_{i\in\NN}$ of the
cover $\{q(O_{x})\}_{x\in G_{0}}$. Pulling back the sets
$\{V'_{i}\}_{i\in\NN}$ to $G_{0}$ we get a locally finite covering
$\{V_{i}\}_{i\in\NN}$ of the space $G_{0}$ by $G$-invariant open
subsets, where we denoted $V_{i}=q^{-1}(V_{i}')$. For each
$i\in\NN$ we can choose some $x_{i}$, such that $V_{i}\subset
O_{x_{i}}$, to get a faithfull unitary representation of the
groupoid $G|_{V_{i}}$ on the hermitian vector bundle
$E_{V_{i}}=E_{O_{x_{i}}}|_{V_{i}}$. By Proposition \ref{Extension}
we can extend the unitary representation of the groupoid
$G|_{V_{i}}$ on the bundle $E_{V_{i}}$ to the unitary
representation of the groupoid $G$ on the family of finite
dimensional Hilbert spaces $E^{i}=E_{V_{i}}^{G_{0}}$.

Let $E=\bigoplus_{i\in\NN}E^{i}$ be the continuous family of
finite dimensional Hilbert spaces over $G_{0}$, defined as the sum
of the families $\{E^{i}\}_{i\in\NN}$ as in Example
\ref{Ex_Families of Hilbert spaces}. The representations of the
groupoid $G$ on the families $\{E^{i}\}_{i\in\NN}$ canonically
induce a continuous unitary representation of the groupoid $G$ on
$E$, defined by $g\cdot(v_{1},v_{2},\ldots)=(g\cdot v_{1},g\cdot
v_{2},\ldots)$. To see that the representation of $G$ on $E$ is
faithfull it is enough to show that for each $x\in G_{0}$ the
group $G_{x}$ acts faithfully on the Hilbert space $E_{x}$.
Straight from the definition of the representation of $G$ on $E$
it follows that the representation of the group $G_{x}$ on $E_{x}$
decomposes as the direct sum of the representations of the group
$G_{x}$ on the spaces $E^{i}_{x}$ for $i\in\NN$. Since
$\{V_{i}\}_{i\in\NN}$ is a cover of the space $G_{0}$, there
exists some $i\in\NN$ such that $x\in V_{i}$. Faithfullness of the
representation of the groupoid $G|_{V_{i}}$ on the bundle
$E_{V_{i}}$ implies that the representation of the group $G_{x}$
on $E^{i}_{x}$ is faithfull and consequently the representation of
the group $G_{x}$ on $E_{x}$ is faithfull as well.
\end{proof}

\section{Orbifolds as global quotients}

\subsection{Families of unitary frames and proper bundles of topological groups}

Let $X$ be a locally compact Hausdorff space, $O\subset X$ an open
subset and $E_{O}$ an $n$-dimensional hermitian vector bundle over
$O$. Denote by $E_{O}^{X}$ the trivial extension of the hermitian
vector bundle $E_{O}$ to a continuous family of finite dimensional
Hilbert spaces over $X$ as in Example \ref{Ex_Families of Hilbert
spaces}. To the continuous family $E_{O}^{X}$ of Hilbert spaces
over $X$ one can assign a family $\U(E_{O}^{X})$ of unitary frames
over $X$ as follows.

We first recall the definition of the principal $U(n)$-bundle of
unitary frames $\U(E)$ of a hermitian vector bundle $E$ over $B$.
A unitary frame at a point $x\in B$ is an ordered orthonormal base
of the Hilbert space $E_{x}$. We can represent it as a unitary
isomorphism $e_{x}:\CC^{n}\ra E_{x}$, where $\CC^{n}$ is equipped
with the standard scalar product. The set $\U(E)_{x}$ of all
frames of the bundle $E$ at $x$ is equipped with a natural right
action of the Lie group $U(n)$: a group element $A\in U(n)$ acts
on the frame $e_{x}\in\U(E)_{x}$ by $e_{x}\cdot A=e_{x}\circ A$ to
give a new frame at $x$. The bundle $\U(E)$ of unitary frames of
$E$ is the disjoint union of all the spaces $\U(E)_{x}$ with the
natural projection map $\pi:\U(E)\ra B$, sending each of the sets
$\U(E)_{x}$ to their respective $x\in B$. A unitary local
trivialisation $\phi_{i}:E|_{U_{i}}\ra U_{i}\times\CC^{n}$ of the
hermitian vector bundle $E$ induces a local trivialisation
$\psi_{i}:\pi^{-1}(U_{i})\ra U_{i}\times U(n)$ of the bundle
$\U(E)$, given by $\psi_{i}(e)=(\pi(e),\phi_{i,\pi(e)}\circ e)$,
where $\phi_{i,\pi(e)}$ is the unitary isomorphism from
$E_{\pi(e)}$ to $\CC^{n}$. The topology on the space $\U(E)$ is
the finest topology which makes all of the maps $\psi_{i}^{-1}$
continuous.

The definition of the principal $U(n)$-bundle $\U(E_{O})$ of
unitary frames of the hermitian vector bundle $E_{O}$ can be
extended to define the family $\U(E_{O}^{X})$ of unitary frames of
the trivial extension $E_{O}^{X}$ of the bundle $E_{O}$. As a set
$\U(E_{O}^{X})$ is defined to be the disjoint union
\[
\U(E_{O}^{X})=\U(E_{O})\coprod (X\backslash O).
\]
Let $\pi=\pi_{E_{O}}\coprod id|_{X\backslash O}:\U(E_{O}^{X})\ra
X$ denote the projection from the space $\U(E_{O}^{X})$ onto $X$,
where $\pi_{E_{O}}:\U(E_{O})\ra O$ is the ordinary projection from
the bundle of the unitary frames of $E_{O}$ onto $O$. We will
define the topology on the space $\U(E_{O}^{X})$ by specifying its
basis $\cB$. The basic open sets of the space $\U(E_{O}^{X})$ are
of two kinds:
\begin{enumerate}
\item For each open subset $O'$ of $X$ we have
$\pi^{-1}(O')\in\cB$; \item If $O'\subset\U(E_{O})$ is an open
subset then $O'\in\cB$.
\end{enumerate}
Equipped with the topology defined by the basis $\cB$ the family
of frames $\U(E_{O}^{X})$ becomes a locally compact Hausdorff
space such that the map $\pi$ is a continuous open surjection.

Now let $G$ be an orbifold groupoid over $G_{0}$ and let $O$ be a
$G$-invariant open subset of $G_{0}$. To every $n$-dimensional
hermitian vector bundle $E_{O}$ over $O$ we associate the proper
bundle $U_{O}(n)$ of topological groups over the space $Q=G_{0}/G$
in the following way. The fiber of the bundle $U_{O}(n)$ at $x\in
Q$ is given by
\begin{displaymath}
(U_{O}(n))_{x}=\left\{ \begin{array} {ll}
U(n), & \textrm{$x\in q(O)$,}\\
\{0\}, & \textrm{otherwise,}\\
\end{array}\right.
\end{displaymath}
where $q:G_{0}\ra Q$ is the quotient map. The topology on the
space $U_{O}(n)$ is the quotient topology from the space $Q\times
U(n)$, where the $U(n)$-fibers are fibrewise shrunk to a point for
the points outside of $q(O)$. The map $r_{O}:U_{O}(n)\ra Q$ is a
proper continuous map from which it follows that $U_{O}(n)$ is a
proper bundle of topological groups.

We have a natural right action of the proper bundle of groups
$U_{O}(n)$ on the family of unitary frames
$\pi:\U(E_{O}^{G_{0}})\ra G_{0}$ along the map
$q\circ\pi:\U(E_{O}^{G_{0}})\ra Q$. It is explicitely given by the
formula $e_{x}\cdot A_{q(x)}=e_{x}\circ A_{q(x)}$, for $e_{x}\in
\pi^{-1}(O)$ and where $A_{q(x)}\in U(n)$ is seen as a unitary
isomorphism $A_{q(x)}:\CC^{n}\ra\CC^{n}$. For $x$ outside of $O$
the action is defined in the only possible way. Note that the
proper bundle of topological groups $U_{O}(n)$ acts freely and
transitively along the fibers of the map $\pi$.

\subsection{Presenting orbifolds as translation groupoids}

Let $M$ be a smooth manifold and $K$ a compact Lie group acting
smoothly and almost freely on the manifold $M$ from the right. The
translation groupoid $M\rtimes K$ is then Morita equivalent to an
orbifold groupoid. The following proposition shows that the same
is true if we replace the compact group $K$ with some proper
bundle of Lie groups $U$.

\begin{prop}\label{Prop_Translation groupoids of proper bundles of groups}
Let $M$ be a smooth manifold and let $U$ be a proper bundle of Lie
groups over $N$, acting smoothly and almost freely from the right
on the space $M$ along the smooth map $\phi:M\ra N$. Then the
translation groupoid $M\rtimes U$ is Morita equivalent to an
orbifold groupoid.
\end{prop}

\begin{proof}
The idea of the proof is similar to the case of an almost free
action of a compact Lie group. A proper bundle of Lie groups $U$
over $N$ is a bundle of topological groups over $N$ with a
structure of a Lie groupoid. By definition the action of $U$ on
$M$ is almost free if and only if the isotropy groups of the
groupoid $M\rtimes U$ are finite and thus discrete. The groupoid
$M\rtimes U$ is a proper Lie groupoid as a translation groupoid of
a proper Lie groupoid. By Proposition $5.20$ in \cite{MoMr} the
groupoid $M\rtimes U$ is Morita equivalent to an \'etale Lie
groupoid $G$. Since properness is invariant under Morita
equivalence the groupoid $G$ is proper and \'etale, thus an
orbifold groupoid.
\end{proof}

As proved in Theorem \ref{Representation of orbifold groupoids}
each orbifold groupoid $G$ admits a faithfull unitary
representation on a continuous family of finite dimensional
Hilbert spaces over $G_{0}$. We can use Theorem
\ref{Representation of orbifold groupoids} to prove the following
partial converse of Proposition \ref{Prop_Translation groupoids of
proper bundles of groups}.

\begin{theo} \label{Orbifold groupoids Morita equivalence}
Let $G$ be an orbifold groupoid. Then $G$ is Morita equivalent to
a translation groupoid associated to a continuous almost free
action of a proper bundle of topological groups on a topological
space. The bundle can be chosen such that the fibers are finite
products of unitary groups.
\end{theo}

We will prove Theorem \ref{Orbifold groupoids Morita equivalence}
by constructing a space $\pi:\U(E)\ra G_{0}$ over $G_{0}$,
equipped with a free left action of the orbifold groupoid $G$
along the map $\pi$ and with a right action of a proper bundle of
topological groups that acts freely and transitively along the
fibers of the map $\pi$.

\begin{prop}\label{Prop_Unitary rep to continuous action}
Let $G$ be an orbifold groupoid over $G_{0}$, $O$ a $G$-invariant
open subset of $G_{0}$ and let $E_{O}^{G_{0}}$ denote the trivial
extension of the hermitian vector bundle $E_{O}$ over $O$ to a
continuous family of finite dimensional Hilbert spaces over
$G_{0}$. Every continuous unitary representation of the groupoid
$G$ on the family $E_{O}^{G_{0}}$ induces a continuous action of
the groupoid $G$ on the family of frames $\U(E_{O}^{G_{0}})$.
\end{prop}

\begin{proof}
Define the action of the groupoid $G$ on the space
$\U(E_{O}^{G_{0}})$ as follows:
\begin{enumerate}
\item For $e_{x}\in\U(E_{O})_{x}$ and $g\in G(x,y)$ define $g\cdot
e_{x}=g\circ e_{x}$, where $g$ on the right is interpreted as a
unitary map from $E_{x}$ to $E_{y}$, coming from the
representation of the groupoid $G|_{O}$ on the bundle $E_{O}$.

\item For $x\in G_{0}\backslash O$ and $g\in G(x,y)$ define
$g\cdot x=y$.
\end{enumerate}
To show that this defines a continuous action
$\mu:G\times_{G_{0}}\U(E_{O}^{G_{0}})\ra\U(E_{O}^{G_{0}})$ we use
similar techniques as in the proof of Proposition \ref{Extension}.
First decompose the space $\U(E_{O}^{G_{0}})$ as a disjoint union
of the subspaces $G_{0}\backslash\overline{O}$, $\partial O$ and
$\U(E_{O})$.

The sets $G_{0}\backslash\overline{O}$ and $\U(E_{O})$ are basic
open subsets of the space $\U(E_{O}^{G_{0}})$. First note that the
restriction of the map $\mu$ to the open set
$G\times_{G_{0}}(G_{0}\backslash\overline{O})$ is equal to the
natural left action of the groupoid $G$ on
$G_{0}\backslash\overline{O}$ and thus continuous. Choose now any
element $(g,e)\in G\times_{G_{0}}\U(E_{O})$, where $g\in G|_{O}$
is an arrow from $x$ to $y$, and unitary local trivializations of
the vector bundle $E_{O}$ around $x$ respectively $y$. The unitary
representation of the groupoid $G$ on $E_{O}$ induces a continuous
map $m_{g}$ from a small neighbourhood of the arrow $g$ into the
group $U(n)$, with respect to these two local trivializations. In
the associated principal bundle charts the action of $G$ on
$\U(E_{O})$ then looks like multiplication by the map $m_{g}$ and
is hence continuous.

It remains to be proven that $\mu$ is continuous at the points of
the form $(g,x)\in G\times_{G_{0}}\U(E_{O}^{G_{0}})$ where
$x\in\partial O\subset \U(E_{O}^{G_{0}})$ and $g\in G(x,y)$. We
then have $g\cdot x=y$ and $y\in\partial O$ as well. Let $W$ be
any neighbourhood of the point $y\in\partial O
\subset\U(E_{O}^{G_{0}})$. By the definition of the topology on
the space $\U(E_{O}^{G_{0}})$ there exists a neighbourhood $V$ of
the point $y\in G_{0}$ such that $\pi^{-1}(V)\subset W$, where
$\pi:\U(E_{O}^{G_{0}})\ra G_{0}$ is the projection map onto
$G_{0}$. Choose a bisection $V_{g}$ of the arrow $g\in G$ such
that $t(V_{g})\subset V$. The set
$V_{g}\times_{G_{0}}\U(E_{O}^{G_{0}})$ is then an open
neighbourhood of the point $(g,x)$ such that
$\mu(V_{g}\times_{G_{0}}\U(E_{O}^{G_{0}}))\subset\pi^{-1}(V)\subset
W$.
\end{proof}

Now choose an orbifold groupoid $G$ over $G_{0}$ and let $E$ be a
continuous family of finite dimensional Hilbert spaces over
$G_{0}$ together with a faithfull unitary representation of the
groupoid $G$ as constructed in the proof of Theorem
\ref{Representation of orbifold groupoids}. Here we use the same
notations. The family $E$ can be decomposed as a direct sum
$E=\bigoplus_{i\in\NN}E^{i}$, where each $E^{i}=E_{V_{i}}^{G_{0}}$
is a trivial extension of the hermitian vector bundle $E_{V_{i}}$
over $V_{i}$. We define the family of unitary frames $\U(E)$ of
the family $E$, with respect to the decomposition
$E=\bigoplus_{i\in\NN}E^{i}$, to be the fibrewise product of the
families $\{\U(E^{i})\}_{i\in\NN}$ along the projection maps
$\pi_{i}$,
\[
\U(E)=\{(e_{1},e_{2},\ldots)\in\prod_{i\in\NN}\U(E^{i})|\pi_{1}(e_{1})=\pi_{2}(e_{2})
=\ldots\}.
\]
The space $\U(E)$ has a natural projection $\pi$ onto the space
$G_{0}$, induced from any of the projections $\pi_{i}$. The fiber
of the space $\U(E)$ over a point $x\in G_{0}$ can be canonically
identified with the product of the spaces of frames
$\prod_{i\in\NN}\U(E^{i})_{x}$. Since the family
$\{V_{i}\}_{i\in\NN}$ is a locally finite cover of the space
$G_{0}$ this product is in fact finite and therefore homeomorphic
to a finite product of unitary groups.

\begin{prop}
The faithfull unitary representation of the orbifold groupoid $G$
on the continuous family of Hilbert spaces $E$ over $G_{0}$
induces a continuous free action of the groupoid $G$ on the family
of unitary frames $\U(E)$.
\end{prop}

\begin{proof}
The action of the orbifold groupoid $G$ on the space $\U(E)$ along
the map $\pi$ can be defined coordinatewise by
$g\cdot(e_{1},e_{2},\ldots)=(g\cdot e_{1},g\cdot e_{2},\ldots)$.
The continuity of this action follows from the fact that all the
actions of the groupoid $G$ on the spaces $\U(E^{i})$ are
continuous by Proposition \ref{Prop_Unitary rep to continuous
action} and from the fact that the topology on the space $\U(E)$
is the one induced from the product topology on $\prod_{i\in\NN}
\U(E^{i})$.

Choose an arrow $g\in G$ and an element
$e=(e_{1},e_{2},\ldots)\in\U(E)_{x}$ such that $g\cdot e=e$. Then
$g$ must be an isotropy element, $g\in G_{x}$, and there exists
some $i\in\NN$ such that $x\in V_{i}$. From the definition of the
bundle $E^{i}$ it follows that $G_{x}$ acts faithfully on
$E^{i}_{x}$ and therefore $g\cdot e_{i}=e_{i}$ implies $g=1_{x}$.
This shows that the action of the groupoid $G$ on the space
$\U(E)$ is free.
\end{proof}

Denote by $U_{i}=U_{V_{i}}(n_{i})$ the proper bundle of
topological groups over the space $Q=G_{0}/G$, associated to the
hermitian vector bundle $E_{V_{i}}$ over the $G$-invariant open
subset $V_{i}$ of $G_{0}$. The fibrewise product $U$ of the
bundles $U_{i}$ has a natural structure of a proper bundle of
topological groups over $Q$, with each fiber being isomorphic to a
finite product of unitary groups. The right actions of the bundles
$U_{i}$ on the families of unitary frames $\U(E^{i})$ induce a
right action of the bundle $U$ on the space $\U(E)$, defined by
the formula
$(e_{1},e_{2},\ldots)\cdot(A_{1},A_{2},\ldots)=(e_{1}\cdot
A_{1},e_{2}\cdot A_{2},\ldots)$. It is not hard to see that the
proper bundle of groups $U$ acts freely and transitively along the
fibers of the map $\pi:\U(E)\ra G_{0}$.

\begin{proof}[Proof of Theorem \ref{Orbifold groupoids Morita equivalence}]
For the convenience of the reader we first recall the data we have
so far. Let $G$ be an orbifold groupoid over $G_{0}$ and let
$Q=G_{0}/G$ be the space of orbits of the groupoid $G$. We have
constructed the space of frames $\U(E)$, together with the moment
maps $\pi:\U(E)\ra G_{0}$ (the projection map) and the map
$u=q\circ\pi:\U(E)\ra Q$, where $q:G_{0}\ra Q$ is the quotient
projection. There are actions of the groupoid $G$ and of the
proper bundle of groups $r:U\ra Q$ on the space $\U(E)$ from the
left along the map $\pi$ respectively from the right along the map
$u$. Both of these actions are free and moreover $U$ acts on the
space $\U(E)$ transitively along the fibers of the map $\pi$.

Now observe that both the actions are basically compositions of
linear maps from the left respectively from the right. The
associativity of the composition implies that the actions of $G$
and $U$ on the space $\U(E)$ commute. Combining this with the fact
that the map $u:\U(E)\ra Q$ is $G$-invariant, as shown by the
equalities
\[
u(g\cdot e)=q(\pi(g\cdot e))=q(t(g))=q(s(g))=q(\pi(e))=u(e),
\]
we can define a right action of the proper bundle of groups $U$ on
the quotient space $G\backslash\U(E)$, along the induced map
$u':G\backslash\U(E)\ra Q$, by the formula $[e]\cdot A=[e\cdot A]$
for $[e]\in G\backslash\U(E)$ and $u'([e])=r(A)$. This action is
almost free since the action of $U$ on $U(E)$ was free and since
$G$ has finite isotropy groups.

Let $H$ be the (proper) translation groupoid associated to this
action (see Example \ref{Ex_Groupoids}). It has the quotient
$H_{0}=G\backslash\U(E)$ as the space of objects and the space of
arrows equal to $(G\backslash\U(E))\times_{Q}U$. Note that $H_{0}$
is a Hausdorff space since $\U(E)$ is Hausdorff and $G$ is a
proper groupoid. The multiplication in the groupoid $H$ is defined
by the formula $([e],A)([e'],A')=([e],AA')$ for $[e\cdot A]=[e']$
and $u'([e])=u'([e'])=r(A)=r(A')$. The source and the target maps
of the groupoid $H$ are given by $s([e],A)=[e\cdot A]$
respectively $t([e],A)=[e]$. We have a natural action of the
translation groupoid $H$ on the space $\U(E)$, induced from the
action of the proper bundle of groups $U$ on $\U(E)$ and defined
by $e\cdot([e],A)=e\cdot A$.

We will show that the space of frames $\U(E)$, together with the
moment maps $\pi:\U(E)\ra G_{0}$ and $\phi:\U(E)\ra
H_{0}=G\backslash\U(E)$ (the quotient projection), and the actions
of groupoids $G$ respectively $H$, represents a Morita equivalence
between the orbifold groupoid $G$ and the translation groupoid
$H$.

The translation groupoid $H$ acts along the fibers of the map
$\pi$ because the bundle of groups $U$ does so, while the groupoid
$G$ acts along the fibers of the map $\phi$ by the definition of
$\phi$. Similarly, it is not hard to see that both actions
commute, so it remains to be proven that $\phi:\U(E)\ra H_{0}$ is
a principal left $G$-bundle and that $\pi:\U(E)\ra G_{0}$ is a
principal right $H$-bundle. Both the maps $\phi$ and $\pi$ are
open, the first being the quotient map of a groupoid action and
the second one being open as a projection map of a fibrewise
product along a family of open maps. Since the action of the
groupoid $G$ on the space $\U(E)$ is free and transitive along the
fibers of the map $\phi$ the map
$i_{G}:G\times_{G_{0}}\U(E)\ra\U(E)\times_{H_{0}}\U(E)$, given by
$i_{G}(g,e)=(g\cdot e,e)$, is a continuous bijection. Furthermore,
since the groupoid $G$ is proper, the action of $G$ on $\U(E)$ is
proper so $i_{G}$ is a closed map and hence a homeomorphism. This
proves that $\phi:\U(E)\ra H_{0}$ is a principal left $G$-bundle.
Similarly, the map
$i_{H}:\U(E)\times_{H_{0}}H\ra\U(E)\times_{G_{0}}\U(E)$, defined
by $i_{H}((e,([e],A))=(e,e\cdot A)$, defines a homeomorphism which
shows that $\pi:\U(E)\ra G_{0}$ is a principal right $H$-bundle.
\end{proof}

\noindent {\bf Acknowledgements.} I would like to thank A.
Henriques, D. McDuff and J. Mr\v{c}un for discussions and advices
related to the paper.

\end{document}